\documentclass[12pt]{amsart}
\usepackage{euscript, amssymb, amsmath, amsthm, mathrsfs}
\usepackage{epsfig}
\usepackage{graphicx}
\usepackage{caption}
\usepackage{longtable}
\usepackage{dcolumn}
\usepackage{setspace}
\usepackage{breqn}
\usepackage{bbm}
\usepackage{mathtools}
\usepackage[dvipsnames,table,xcdraw]{xcolor}
\usepackage{booktabs}
\usepackage{lmodern} 
\let\e\epsilon
\usepackage{color}
\usepackage[pagebackref=true, colorlinks=true, citecolor=blue]{hyperref}
\usepackage{cite}
\numberwithin{equation}{section}
\usepackage[rightcaption]{sidecap}
\usepackage{wrapfig}
\usepackage{float}

\newcommand{\ud}{\mathrm{d}}

\newtheorem{thm}{Theorem}[section]

\newtheorem{lem}{Lemma}[section]

\theoremstyle{remark}

\numberwithin{equation}{section}

\newcommand{\R}{\mathbb{R}}
\newcommand{\bL}{\text{{\bf L}}}
\newcommand{\bH}{\text{{\bf H}}}
\newcommand{\tr}{\text{tr}}
\newcommand{\Pj}{ \text{{\bf P}}}

	\vspace{-0.2in}\title[Phase-field model for thrombus and blood flow]{ Local well-posedness of a three-dimensional phase-field model for thrombus and blood flow}
\author[W. Kim, K. Tawri and R. Temam]{Woojeong Kim$^{1}$, Krutika Tawri$^{1,*}$ and Roger Temam$^{1}$}

\email{ktawri@iu.edu}
\email{wki1@iu.edu }
\email{temam@indiana.edu}
\thanks{*Corresponding author. Email address: ktawri@iu.edu}

\date{}

\begin{document}
	\maketitle
	{\vspace{-0.3in}
		\centering{$^{1}$ Department of Mathematics and Institute for Scientific Computing and Applied Mathematics,}\\
		\centering{Indiana University,}
		\centering{Bloomington, IN 47405, USA.\\}
}
	\begin{abstract}
	In this article we consider a fluid-structure interactions model on a three dimensional bounded domain, that describes the mechanical interaction between blood flow and a thrombus with Hookean elasticity. The interface between the two phases is given by a smooth transition layer, diffuse with a finite thickness. We derive various a priori estimates and prove local well-posedness results using the Faedo-Galerkin method.
	\end{abstract}
Mathematics Classification Code: 35D35, 35G31, 35Q35, 76T30  \\
Keywords: Navier-Stokes equations, Cahn-Hilliard equations, Oldroyd model, local strong solutions, uniqueness.

	\section{Introduction}
	
	Historically, phase-field methods have been commonly used to model two-phase flows of macroscopically immiscible fluids. However recently there have been some efforts in using phase-field methods to model fluid-structure
	interactions (FSI).	In \cite{SXJ} and \cite{MAA}, the authors give a fully Eulerian description of the velocity field in the fluid and the
	elastic domain by coupling Oldroyd-B type equations with volume preserving Allen-Cahn and Cahn-Hilliard type equations respectively. 
	
	In this article we study an FSI problem proposed in \cite{ZYLHK} (see also \cite{YZHK}) to model the mechanical interaction between blood flow and a thrombus with Hookean elasticity, using a phase-field method to describe a two-phase system, with the interface between two different phases given by a thin smooth transition layer. 

Let $x(t,\cdot)$ be a time-dependent family of orientation-preserving
diffeomorphisms. Let $X(x,t)$ be the
corresponding reference map i.e. $X( \cdot,t)$ is the inverse of $x(\cdot,t)$. The velocity field
$u(x, t)$ is defined as:
$$u(x,t)=\frac{\ud x(t,X)}{\ud t}|_{X=X(x,t)}.$$
The deformation gradient $F$ is given by 
$$F(x,t)=\frac{\partial x(t,X)}{\partial X}|_{X=X(x,t)}.$$
Using the chain differentiation rule it can be seen that $F$ satisfies
$$F_t + u\cdot \nabla F=\nabla u F,$$
which is to be understood component-wise as $F^{ij}_t + \sum_{k=1}^d u^k\partial_kF^{ij}=\sum_{k=1}^d\partial_ku^iF^{kj}$ for $1\leq i,j\leq d$. Here we use the notation $F^{ij}=\frac{\partial x^i}{\partial X^j}$ and $[\nabla u]_{ij}:=\frac{\partial u^i}{\partial x^j}$.
The incompressibility is {expressed in the Lagrangian coordinates} by the equation: $$\det F \equiv 1,$$
and it is  expressed in the Eulerian coordinates by the equation div $u=0$ which appears below in \eqref{pde}$_2$. See the explanations in \cite{MT}.
For a bounded domain $\Omega \subset \R^d$ with a sufficiently smooth boundary $\partial \Omega$ and $T>0$, we consider the following governing equations on $\Omega \times (0,T)$
\begin{equation}\begin{split}\begin{cases}\label{pde}
\rho(\frac{\partial u}{\partial t} + u \cdot \nabla u) + \nabla p - \nabla \cdot (\eta(\phi)\nabla u) = -\lambda \nabla \cdot (\nabla \phi \otimes \nabla \phi) \\
\hspace{2.2in}+\nabla \cdot (\lambda_e(1-\phi)(FF^T-I))-\eta(\phi)\frac{(1-\phi)u}{\kappa(\phi)},\\
\nabla \cdot u =0,\\
\frac{\partial F}{\partial t} + u\cdot \nabla F=\nabla u F,\\
\frac{\partial \phi}{\partial t} + u \cdot \nabla \phi =\tau \Delta \mu ,\\
\mu=-\lambda \Delta \phi  + \lambda\gamma f'(\phi)-\frac{\lambda_e}{2}\tr(FF^T-I),
\end{cases}\end{split}\end{equation}	
where $u, p,\rho,\eta, \kappa$ are the velocity, pressure, mass density, dynamic viscosity and
permeability, respectively, and $\phi \in [0,1], \mu$ are the phase-field variable and the so-called chemical potential respectively. The phase-field variable $\phi = 1$ represents
the blood, $0 < \phi < 1$ represents a mixture of blood and thrombus, and $\phi = 0 $ represents the
thrombus. The parameters 
$\gamma, \tau, \lambda$ and $\lambda_e$ stand for the interfacial mobility, relaxation parameter,
mixing energy density, and visco-elastic modulus, respectively and are assumed to be positive constants. In this article, we shall take the function $f$ to be the double-well potential $f(\phi)= \frac{(\phi-1)^2\phi^2}{4h^2},$ where $h$ is the interfacial thickness. The above equations \eqref{pde} are subjected to the following initial and boundary conditions
\begin{equation}\label{bc}
\begin{split}
u=0, \partial_{\bf n}\mu= \partial_{\bf n}\phi&=0 \quad \text{on } \partial \Omega \times (0,T),\\
u(\cdot,0) = u_0, \, \phi(\cdot,0)&=\phi_0,\, F(\cdot,0)=F_0 \quad \text{in } \Omega,
\end{split}
\end{equation}
where ${\bf n}$ is the unit vector normal to the boundary $\partial \Omega$, pointing outward.

Since the works \cite{LLZ05}, \cite{LW} there have been several results concerning the viscoelastic systems. In \cite{GS}, the authors studied the case of \eqref{pde} with $\phi=0$ and an additional linear damping term in the $F$-equation, and proved the global existence of a small smooth solution. When the contribution of the symmetric part of $\nabla u$ in the constitutive equation is neglected, the global existence of weak solutions with general initial data was proven in \cite{LM}. In case of bounded domain, the global existence of strong solutions for small initial data was proven in \cite{LZ08}.

The Cahn-Hilliard-Navier-Stokes system, to be precise the model H (see e.g. \cite{AMW}, \cite{HH77},  \cite{GPV}) obtained by setting $F=0$, has been studied extensively. The mathematical analysis of the system is well established for classical boundary conditions, specially in the case of regular nonlinear terms $f$. In three dimensions, the local existence and uniqueness of strong solutions is known whereas in two dimensions the existence of a global strong solution is known. We refer the reader to \cite{St97}, \cite{Bo99}, \cite{Ab09}, \cite{GMT}.

For the purpose of our analysis, we will assume that $\eta, \kappa \in C^2(\R)$ and that for some $\alpha,\beta>0$, \begin{align}\label{boundseta}\alpha \leq \eta(x), \eta'(x), \kappa(x) \leq \beta \quad \forall x \in \R.\end{align}
Next we will introduce the functional framework required for our analysis.\\
Consider the space of solenoidal vector fields $\mathcal{V} =\{u\in {C^\infty_0}({\Omega})^d, \text{div }u=0\}$. The closures of $\mathcal{V}$ in $L^2(\Omega)^d$ and $H^1_0(\Omega)^d$ are called $H$ and $V$ respectively, and they are characterized as follows (see \cite{T_NSE}): 
  $$H=\{ u \in L^2({\Omega})^d, \text{ div }u=0, u \cdot {\bf n} =0 \text{ on } \partial {\Omega}\}$$ and
   $$V=\{u \in H^1_0({\Omega})^d, \text{ div } u =0\}.$$
We denote by $D(A) \subset V$ the domain of the Stokes operator $A:=-\mathbb{P}\Delta$, where $\mathbb{P}$ is the Helmholtz-Leray orthogonal projection from $L^2(\Omega)^d$ onto $H$; we know that $D(A)=H^2(\Omega)^d\cap V$.

We will equip the space $L^2(\Omega)$ with norm $|\cdot|$ and inner-product $(\cdot,\cdot)$; same notation for the spaces $L^2(\Omega)^d$ and $L^2(\Omega)^{d\times d}$. For $\Xi=[\xi_{ij}] \in L^2(\Omega)^{d \times d}$ where $1\leq i,j \leq d$ the norm is $|\Xi|=\left(\sum_{ij}\int_\Omega \xi_{ij}^2\ud x\right)^{\frac12}$ which is also equal to $\left(\int_\Omega \tr(\Xi\Xi^T)\ud x\right)^{\frac12}$. The associated inner product is given by $(\Xi_1,\Xi_2)=\int_\Omega \tr(\Xi_1\Xi_2^T)\ud x$ for any $\Xi_1,\Xi_2 \in L^2(\Omega)^{d \times d}$.

Our main result concerning the local well-posedness of the system \eqref{pde}-\eqref{bc} is stated in the following theorem.
\begin{thm}\label{main}
	For  $d=2,3$, let $\Omega \subset \R^d$ be a bounded open set with a sufficiently smooth boundary. We are given $u_0 \in D(A), \phi_0 \in H^3(\Omega)$ such that $\partial_{\bf n}\phi_0=0$ and $F_0 \in H^2(\Omega)^{d\times d}$. Then there exists $0<T_0\leq T$ such that \eqref{pde}-\eqref{bc} has a unique solution $(u,F,\phi, p)$ on $[0,T_0]$ such that 
	\begin{equation}\begin{split} \label{strong}
	&u\in C([0,T_0];D(A)) \cap L^2(0,T_0;H^3(\Omega)^d) ,\quad \partial_t u \in C([0,T_0];H) \cap L^2(0,T_0;V),\\
	&F\in C([0,T_0];H^1(\Omega)^{d\times d}) \cap L^\infty(0,T_0;H^2(\Omega)^{d\times d}),\\
	&\phi \in C([0,T_0];H^3(\Omega)) \cap L^2(0,T_0;H^4(\Omega)) \cap H^1(0,T_0;H^3(\Omega)),\\
	&\nabla p \in L^2(0, T_0 ; H^1(\Omega)).
	\end{split}\end{equation}
\end{thm}
\section{A priori estimates} In this section we perform standard energy estimates on the system \eqref{pde} in order to apply the Faedo-Galerkin scheme in Section \ref{galerkin}. We will only consider the case $d=3$; however the same results can be obtained more easily for the case when $d=2$. Hereupon, $C$ will represent a generic positive constant depending on $\Omega$, $\alpha$ and $\beta$, unless specified otherwise. We will use the short hand $|.|_s$ to denote $|.|_{H^s(\Omega)}$ or $|.|_{H^s(\Omega)^d}$ or $|.|_{H^s(\Omega)^{d\times d}}$ depending on the argument, where $H^s(\Omega)$, $s \in \R_{+}$ is the Sobolev space of order $s$; see e.g. \cite{LM72}. We will also use boldface letters to denote spaces containing vector-valued functions. 

Observe that, unlike the Cahn-Hilliard equations or the usual model-H, for our model we do not have that $\partial_{\bf n}\Delta\phi = 0$ on the boundary $\partial\Omega$ due to the form of \eqref{pde}$_5$  and \eqref{bc}. Instead we have $\partial_{\bf n}\Delta \phi=-\frac{\lambda_e}{2}\partial_{\bf n}\tr(FF^T)$; see below. Hence we will frequently use the following consequence of the {regularity results of the Agmon-Douglis-Nirenberg type \cite{ADN}}.
\begin{lem}\label{lem_bi}
Let $\Omega \subset \R^d$ be a bounded open set with a sufficiently smooth boundary. Assume that for $f \in L^2(\Omega)$ and $g \in H^{\frac12}(\partial\Omega)$, $\phi \in H^2(\Omega)$ are solutions, in a weak sense, of the following biharmonic inhomogeneous Neumann boundary value problem,
	\begin{equation}\begin{split}\label{biprob}
	\Delta^2\phi &=f \quad \text{ in } \Omega,\\
	\partial_{\bf n} \phi&=0,\,\, \partial_{\bf n}\Delta \phi =g \quad \text{ on }\partial\Omega,
	\end{split}	\end{equation}
	where $f$ and $g$ satisfy the following compatibility condition:
	$$\int_\Omega f\ud x = \int_{\partial \Omega} g \ud \Gamma.$$
	Then $\phi \in H^4(\Omega)$ and there exists some constant $\tilde{C}>0$ independent of $\phi,f,g$ such that
	\begin{align}\label{biresult}
	|\phi|^2_{H^{4}(\Omega)/\R} \leq \tilde{C} \left(|f|^2 + |g|^2_{H^{\frac12}(\partial\Omega)} \right), \text{ or}\\
	|\phi|^2_{H^{4}(\Omega)} \leq \tilde{C} \left(|f|^2 + |g|^2_{H^{\frac12}(\partial\Omega)} + |\int_\Omega\phi\ud x|^2\right).
	\end{align}
\end{lem}
\begin{proof}
	The Lax-Milgram theorem shows that $\phi \in H^2(\Omega)$, a weak solution of \eqref{biprob}, exists and is unique up to an additive constant. Then $\phi \in H^4(\Omega)$ follows from the results of Agmon-Douglis-Nirenberg \cite{ADN} and {the first inequality} \eqref{biresult} holds. The additive constant is accounted for by the term $|\int_\Omega\phi\ud x|$.
\end{proof}

{Similar to \eqref{biresult} we have that if $\phi \in H^{2}(\Omega)$ is such that $\partial_{\bf n}\phi=0$ on $\partial\Omega$, then, for some constant $C>0$,}
\begin{align}\label{equivnorm}|\phi|_{2} \leq C\left(|\Delta \phi| + |\int_\Omega \phi \ud x|\right).
\end{align}
For brevity, we will use the following notation for the mean value of a function $\phi$, 
$$\langle \phi \rangle_{\Omega} :=\frac1{|\Omega|}\int_\Omega \phi \ud x.$$
Observe that by integrating equation \eqref{pde}$_4$ over $\Omega$ and using the divergence free property of $u$ along with the boundary conditions \eqref{bc} we have 
\begin{align}\label{avs}\langle\phi_t\rangle_\Omega=\langle\Delta \phi\rangle_\Omega=\langle\Delta \phi_t\rangle_\Omega= 0.\end{align}
This further implies that for any time $t>0$, \begin{align}\label{K0}|\langle \phi(t) \rangle_\Omega|=|\langle \phi_0 \rangle_\Omega|=:K_0. \end{align}

	We will also use the following lemma to obtain estimates on $u$ in the space norm $\bH^3(\Omega)$ later in this section. See \cite{CK04} for a similar result in the context of the Navier-Stokes equations with density-dependent viscosity.
	\begin{lem}\label{Stokes}
	Let $\Omega \subset \R^3$ be an open bounded set  with a sufficiently smooth boundary. Assume that  $ (u, p ) \in V\times L^2(\Omega)$ is the weak solution of the following problem
	\begin{equation}\begin{split}\label{stokeseqn}
	-\text{ div }( \eta(\phi)\nabla u) + \nabla p&=f, \quad {\text {in } } \Omega,\\
	\text{ div }u&=0\quad \text{in }\Omega,\\
	\int_\Omega \frac{p}{\eta(\phi)} \ud x&=0,
\end{split}	\end{equation}
where $\eta \in C^2(\R)$ satisfies the assumption \eqref{boundseta} and $\phi \in H^3(\Omega)$. Then there exists a constant $C>0$ such that
	\begin{align}\label{stokesresult}	|u|_3 + \left|\frac{p}{\eta(\phi)}\right|_1 \leq C|f|_1  \left[ 1+ \left( 1+ | \phi|_{2}^2\right) \left( |\phi|^2_2+ |\phi|_2^{\frac12}|\phi|_3^{\frac12}\right)  \right] .
	\end{align}	\end{lem}
	\begin{proof}[Proof of Lemma \ref{Stokes}]
	Firstly, we know that the divergence operator maps the space $\bH^1_0(\Omega)$ onto the space ${{L}^2(\Omega)/ \R}:=\{g \in L^2(\Omega); \int_\Omega g \ud x=0\}$ (see Lemma 2.4 in \cite{T_NSE}). Hence we have the existence of $ \xi \in \bH^1_0(\Omega)$ such that,
		$$\text{div }\xi= \frac{p}{\eta(\phi)},$$
		where $\xi$, which is not unique, can be chosen such that 
		\begin{align}| \xi|_{1} \leq C\left|\frac{p}{\eta(\phi)}\right|.\end{align}
	For a proof of the above result see e.g. \cite{Galdi} (III.3.2-III.3.3).	\\
	Testing \eqref{stokeseqn} with $\xi$ and using the standard estimate $|u|_1 \leq \alpha^{-1}|f|$ we obtain,
		\begin{equation}\begin{split}\left|\frac{p}{\sqrt{\eta(\phi)}}\right|^2 =(\eta(\phi)\nabla u, \nabla \xi) - (f,\xi) &\lesssim \beta|u|_1|\xi|_1 +|f||\xi|_1 \\
		&\lesssim (1+ \beta\alpha^{-1})|f||\frac{p}{\eta(\phi)}|.\label{p1}
		\end{split} \end{equation}
		 Additionally observe that,
		 \begin{align}\alpha  \left|\frac{p}{{\eta(\phi)}}\right|^2 \leq \left|\frac{p}{\sqrt{\eta(\phi)}}\right|^2.\label{p2} \end{align}
		 Combining \eqref{p1} and \eqref{p2} we obtain,		 
		 \begin{align}\label{pl2}\left|\frac{p}{{\eta(\phi)}}\right| \leq C(\alpha,\beta) |f|.
		 \end{align}
		 Now observe that \eqref{stokeseqn}$_1$ can be re-written as,
		\begin{equation}
		\begin{split}
		-\Delta u + \nabla(\frac{p}{\eta(\phi)})= \frac{f}{\eta(\phi)} + \frac{\nabla u\cdot \nabla \eta(\phi)}{\eta(\phi)}- \frac{p\nabla \eta(\phi)}{\eta(\phi)^2}.
		\end{split}
		\end{equation}
		Applying the classical result on the regularity of the Stokes problem (see \cite{Ca61}) we have,
		\begin{align*}
			|u|_2 + \left|\nabla \frac{p}{\eta(\phi)}\right|& \lesssim \left|\frac{f}{\eta(\phi)}\right| + \left|\frac{\nabla u\cdot \nabla \eta(\phi)}{\eta(\phi)}\right|+ \left|\frac{p\nabla \eta(\phi)}{\eta(\phi)^2}\right|.
		\end{align*}
		We treat each term on the right hand side of the above inequality using the H\"older inequality and the Sobolev embeddings $H^{\frac12}(\Omega) \subset L^3(\Omega)$ and $H^1(\Omega)\subset L^6(\Omega)$ in the following way:
		\begin{align*}
		\left|\frac{\nabla u\cdot \nabla \eta(\phi)}{\eta(\phi)}\right|=\left|\frac{\eta'(\phi)\nabla u\cdot \nabla \phi}{\eta(\phi)}\right| &\lesssim \alpha^{-1}|\eta'|_{L^\infty}|\nabla u|_{\bL^3(\Omega)}|\nabla \phi|_{\bL^6(\Omega)}\\
		& \lesssim \alpha^{-1}\beta|u|_1^{\frac12}|u|^{\frac12}_2| \nabla \phi|_{\bL^6(\Omega)}
		\shortintertext{and,}
		\left|\frac{p\nabla \eta(\phi)}{\eta(\phi)^2}\right|=\left|\frac{p\,\eta'(\phi)\nabla \phi}{\eta(\phi)^2}\right| &\lesssim \alpha^{-1}|\eta'|_{L^\infty} \left|\frac{p}{\eta(\phi)}\right|_{L^3}|\nabla \phi|_{\bL^6(\Omega)}\\
		& \lesssim \alpha^{-1}\beta\left|\frac{p}{\eta(\phi)}\right|^{\frac12}\left|\nabla\frac{p}{\eta(\phi)}\right|^{\frac12}|\nabla \phi|_{\bL^6(\Omega)}.
		\end{align*}
		Hence using the above bounds and the Schwarz inequality we obtain, for some $C>0$ depending on $\alpha,\beta$, that
			\begin{align*}
		|u|_2 + \left|\nabla \frac{p}{\eta(\phi)}\right|
		& \lesssim \alpha^{-1}|f| + \alpha^{-1}\beta|u|_1^{\frac12}|u|^{\frac12}_2| \nabla \phi|_{\bL^6(\Omega)} + \alpha^{-1}\beta\left|\frac{p}{\eta(\phi)}\right|^{\frac12}\left|\nabla\frac{p}{\eta(\phi)}\right|^{\frac12}|\nabla \phi|_{\bL^6(\Omega)}\\
		&  \leq \frac12|u|_2 + \frac12 \left|\nabla\frac{p}{\eta(\phi)}\right| +\alpha^{-1}|f| + C|u|_1| \nabla \phi|^2_{\bL^6(\Omega)} + C\left|\frac{p}{\eta(\phi)}\right||\nabla \phi|^2_{\bL^6(\Omega)}.
		\end{align*}
	We now use \eqref{pl2} to obtain,
			\begin{equation}
		\begin{split}\label{stokesu2}
		|u|_2 + \left|\nabla \frac{p}{\eta(\phi)}\right|
		& \leq C(\alpha,\beta) \left( 1+ |\nabla \phi|_{\bL^6(\Omega)}^2\right) |f|.
		\end{split}
		\end{equation}
		Similarly using the classical regularity results for the Stokes problem (see \cite{Ca61}) at the next order, we write
		\begin{align*}
		|u|_3 + \left|\nabla \frac{p}{\eta(\phi)}\right|_1 &\lesssim \left|\frac{f}{\eta(\phi)}\right|_1 + \left|\frac{\eta'(\phi)\nabla u\cdot  \nabla \phi}{\eta(\phi)}\right|_1+ \left|\frac{\eta'(\phi)p\nabla \phi}{\eta(\phi)^2}\right|_1
			\\& \lesssim \alpha^{-1}|f|_1 + \alpha^{-1}\beta\sum_{|s|\leq 1}\left( |\partial^s \nabla u|_{\bL^3(\Omega)}|\nabla \phi|_{\bL^6(\Omega)} + |\nabla u|_{\bL^6(\Omega)}|\partial^s\nabla \phi|_{\bL^3(\Omega)}\right) \\
		&\qquad+ \alpha^{-1}\beta\sum_{|s|\leq 1} \left(\left|\nabla \frac{p}{\eta(\phi)}\right|_{\bL^3(\Omega)}|\nabla \phi|_{\bL^6(\Omega)} + \left| \frac{p}{\eta(\phi)}\right|_{L^6(\Omega)}|\partial^s \nabla \phi|_{\bL^3(\Omega)} \right)\\
	& \lesssim  \alpha^{-1}|f|_1 + \alpha^{-1}\beta\left( |  u|_{2}^{\frac12}|u|_3^{\frac12}| \phi|_{2} + |u|_2^{}| \phi|_{2}^{\frac12}|\phi|_3^{\frac12}\right) \\
		&\qquad+ \alpha^{-1}\beta\left(\left|\nabla \frac{p}{\eta(\phi)}\right|^{\frac12}\left|\nabla \frac{p}{\eta(\phi)}\right|_{1}^{\frac12}| \phi|_{2} + \left|\nabla \frac{p}{\eta(\phi)}\right|^{}|  \phi|^{\frac12}_{2}|\phi|_3^{\frac12} \right).\end{align*}Using the Schwarz inequality we can find further upper bounds as follows, \begin{align*}
			& \leq \frac12|u|_3 + \frac12\left|\nabla \frac{p}{\eta(\phi)}\right|_{1}^{}+ \alpha^{-1}|f|_1 + C |u|_{2}| \phi|^2_{2} + |u|_2| \phi|_{2}^{\frac12}|\phi|_3^{\frac12} \\
		&\qquad+ C\left(\left|\nabla \frac{p}{\eta(\phi)}\right|^{}| \phi|^2_{2} + \left|\nabla \frac{p}{\eta(\phi)}\right|^{}|  \phi|^{\frac12}_{2}|\phi|_3^{\frac12} \right). 
		\end{align*}
That is we have,	
	\begin{equation}		\begin{split}
		|u|_3 + \left|\nabla \frac{p}{\eta(\phi)}\right|_1 \leq C \left( |f|_1 + \left( |u|_2+\left|\nabla \frac{p}{\eta(\phi)}\right|\right) (|\phi|^2_2 + |\phi|_2^{\frac12}|\phi|_3^{\frac12})   \right). \label{stokesu3_1}
		\end{split}
		\end{equation}	
	Substituting \eqref{stokesu2} in \eqref{stokesu3_1} we obtain the desired inequality,
			\begin{equation}
		\begin{split}\label{stokesu3}
		|u|_3 + \left|\nabla \frac{p}{\eta(\phi)}\right|_1 &\leq C|f|_1  \left[ 1+ \left( 1+ |\nabla \phi|_{\bL^6(\Omega)}^2\right) \left( |\phi|^2_2+ |\phi|_2^{\frac12}|\phi|_3^{\frac12}\right)  \right].
		\end{split}
		\end{equation}
		With these estimates we conclude the proof of Lemma \ref{Stokes}.
	\end{proof}
Now we begin to find appropriate a priori estimates for the system \eqref{pde} by first testing equation \eqref{pde}$_1$ with $u$,
	\begin{align*}
\frac12\frac{\ud}{\ud t}\int_\Omega|u|^2\ud x + &\int_\Omega{\eta(\phi)}|\nabla u|^2\ud x=(\lambda\nabla \cdot (\nabla \phi \otimes \nabla \phi),u) \\
&\quad	\nonumber-\lambda_e(\text{div}((1-\phi)(FF^T-I)),u) - (\frac{\eta(\phi)}{\kappa(\phi)}(1-\phi)u,u) .
\end{align*}
The first term on the right hand side of \eqref{pde}$_1$ can be reformulated as follows,
\begin{align}-\lambda\nabla \cdot (\nabla \phi \otimes \nabla \phi)
&=-\lambda \Delta \phi \nabla\phi -\lambda\nabla (\frac12 |\nabla \phi|^2);\label{mutophi}
\shortintertext{using \eqref{pde}$_5$ further gives,}
&=\mu \nabla \phi + \frac{\lambda_e}{2}\tr(FF^T-I)\nabla \phi-\lambda\gamma\nabla f(\phi) - \lambda\nabla(\frac12|\nabla \phi|^2)\label{tensor}.\end{align}
The gradient terms appearing on the right hand sides of \eqref{mutophi} and \eqref{tensor} will be accounted for as a part of the pressure term as we use these identities interchangeably in the analysis that follows.\\
  Also observe that, for any matrix $\Xi$, an application of integration by parts gives us the following identity:
 $$\int_\Omega \text{div}(\Xi)\, u\ud x=-\int_\Omega \tr(\nabla u \cdot \Xi)\ud x.$$
 Applying this property we see that,
 \begin{align*}\int_\Omega \lambda_e\text{div}((1-\phi)(FF^T-I))u \, \ud x &= -\lambda_e\int_\Omega \tr(\nabla u(FF^T-I)(1-\phi))\ud x\\
  &=-\lambda_e\int_\Omega \tr(\nabla uFF^T)(1-\phi)\ud x + \lambda_e\int_\Omega \tr(\nabla u)(1-\phi)\ud x 
  \shortintertext{since $\tr(\nabla u)=$ div $u=0$, we obtain}
 &=-\lambda_e\int_\Omega \tr(\nabla uFF^T)(1-\phi)\ud x.\end{align*}
	Thus we obtain,
	\begin{align}\label{u_ener}
	\frac12\frac{\ud}{\ud t}\int_\Omega|u|^2\ud x + &\int_\Omega{\eta(\phi)}|\nabla u|^2\ud x=\int_\Omega \mu \nabla \phi \cdot u\ud x +\frac{\lambda_e}{2}\int_\Omega \tr(FF^T)(u \cdot\nabla \phi) \ud x\\
&\quad	\nonumber-\lambda_e\int_\Omega \tr(\nabla uFF^T)(1-\phi)\ud x -\int_\Omega \frac{\eta(\phi)}{\kappa(\phi)}(1-\phi)u^2 \ud x.
	\end{align}
	Next we test \eqref{pde}$_4$ by $\mu$ and obtain
	\begin{align}\label{test_mu}
	\left( \frac{\ud \phi}{\ud t},\mu\right)  + (u \cdot \nabla \phi,\mu) + \tau|\nabla \mu|^2=0.
	\end{align}
We use the expression \eqref{pde}$_5$ for $\mu$ to expand the first term of equation \eqref{test_mu} in the following way,
	\begin{align} \left( \frac{\ud  \phi}{\ud t},\mu \right) &= \int_\Omega
	\left(-\lambda \Delta \phi  + \lambda\gamma f'(\phi)-\frac{\lambda_e}{2}\tr(FF^T-I)\right)\frac{\ud \phi}{ \ud t} \ud x
	 \nonumber\\
	&=\int_\Omega \frac{\ud }{\ud t}\left(\frac{\lambda}2|\nabla \phi|^2 +\lambda \gamma f(\phi)\right)\ud x-\frac{\lambda_e}{2}\left( \tr(FF^T),\frac{\ud \phi}{\ud t}\right) . \label{dphimu}\end{align}
	In the above equation we have also used \eqref{avs} which gives us  $\frac{\lambda_e}{2}\tr(I)\int_\Omega\phi_t\ud x=0$.\\

Combining \eqref{u_ener}, \eqref{test_mu} and \eqref{dphimu} we find
			\begin{equation}\begin{split}\label{u_phi_eqn}
			&\frac12\frac{\ud}{\ud t}\int_\Omega\left(|u|^2+{\lambda}|\nabla \phi|^2 +2\lambda \gamma f(\phi) \right)\ud x+ \int_\Omega({\eta(\phi)}|\nabla u|^2+\tau |\nabla\mu|^2)\ud x\\ &\qquad =\frac{\lambda_e}{2}\int_\Omega \tr(FF^T)\left(\frac{\ud \phi}{\ud t}+u \cdot\nabla \phi\right) \ud x
	-\lambda_e\int_\Omega \tr(\nabla uFF^T)(1-\phi)\ud x \\
	&\qquad -\int_\Omega \frac{\eta(\phi)}{\kappa(\phi)}(1-\phi)u^2 \ud x.
	\end{split}\end{equation}	
 Next we test \eqref{pde}$_3$ with $F(1-\phi)$ i.e. we multiply \eqref{pde}$_3$ by $F^T(1-\phi)$, take the trace and then integrate over $\Omega$.
 \begin{equation}\begin{split}\label{feqn_1}\int_\Omega \tr(\frac{\ud F}{\ud t}F^T)(1-\phi)\ud x+\int_\Omega \tr((u\cdot \nabla F)F^T)(1-\phi)  \ud x
\\ = \int_\Omega\tr(\nabla u FF^T)(1-\phi)\ud x.\end{split}\end{equation}
For the first term in the equation above, we note that
	\begin{align}
	\tr(\frac{\ud F}{\ud t}F^T)= \frac12 \frac{\ud }{\ud t}\tr(FF^T) .\label{matrix_iden1}
	\end{align}
Also observe that,
$$\tr((u\cdot \nabla F)F^T)=\frac12 u\cdot \nabla \tr(FF^T).$$
This further gives us the following identity,
\begin{align}
\int_\Omega\tr((u\cdot \nabla F)F^T)(1-\phi)\ud x&=\frac12 \int_\Omega u\cdot \nabla \tr(FF^T)(1-\phi) \ud x\nonumber\\
&=-\frac12 \int_\Omega u\cdot \nabla \tr(FF^T)\phi +\frac12 \int_\Omega u\cdot \nabla \tr(FF^T). \nonumber\end{align}
Integrating by parts and using that div$u=0$ we see that the second term on the right is $0$; this further gives\begin{align}
\int_\Omega\tr((u\cdot \nabla F)F^T)(1-\phi)\ud x=-\frac12 \int_\Omega u\cdot \nabla \tr(FF^T)\phi.\label{matrix_iden2}
\end{align}
Using \eqref{matrix_iden1} and \eqref{matrix_iden2} in \eqref{feqn_1} we obtain,
	\begin{equation}\begin{split}\label{f_ip}
\frac12\int_\Omega \left(\frac{\ud }{\ud t}\tr(FF^T)\right)(1-\phi)\ud x- \frac12\int_\Omega u\cdot \nabla \tr(FF^T)\phi  \ud x\\
= \int_\Omega\tr(\nabla u FF^T)(1-\phi)\ud x.
\end{split}	\end{equation}
Substituting the right hand side of the equation \eqref{f_ip} in \eqref{u_phi_eqn} we obtain,
\begin{equation*}\begin{split}
\frac12\frac{\ud}{\ud t}\int_\Omega\left(|u|^2+{\lambda}|\nabla \phi|^2 +2\lambda \gamma f(\phi) \right)&\ud x+ \int_\Omega({\eta(\phi)}|\nabla u|^2+\tau |\nabla\mu|^2)\ud x\\ & =\frac{\lambda_e}{2}\int_\Omega \tr(FF^T)\left(\frac{\ud \phi}{\ud t}+u \cdot\nabla \phi\right) \ud x
\\&\hspace{-0.2in}-
\frac{\lambda_e}2\int_\Omega \left(\frac{\ud }{\ud t}\tr(FF^T)\right)(1-\phi)\ud x+\frac{\lambda_e}2\int_\Omega u\cdot \nabla \tr(FF^T)\phi  \ud x \\
& -\int_\Omega \frac{\eta(\phi)}{\kappa(\phi)}(1-\phi)u^2 \ud x.
\end{split}\end{equation*}	
We will now simplify the first three integrals on the right hand side. Observe that,
\begin{align*}
&\frac{\lambda_e}{2}\int_\Omega \tr(FF^T)\left(\frac{\ud \phi}{\ud t}+u \cdot\nabla \phi\right) -\left(\frac{\ud }{\ud t}\tr(FF^T)\right)(1-\phi)+ u\cdot \nabla \tr(FF^T)\phi \, \ud x\end{align*}\begin{align*}
&= -\frac{\lambda_e}{2} \int_\Omega \frac{\ud}{\ud t} (\tr(FF^T)(1-\phi)) \ud x + \frac{\lambda_e}{2} \int_\Omega \tr(FF^T)u\cdot \nabla \phi + u\cdot \nabla \tr(FF^T)\phi\ud x\\
&= -\frac{\lambda_e}{2} \int_\Omega \frac{\ud}{\ud t} (\tr(FF^T)(1-\phi)) \ud x + \frac{\lambda_e}{2} \int_\Omega \text{div}(\tr(FF^T)u\phi) \ud x.
\shortintertext{Using the divergence theorem we see that $\int_\Omega \text{div}(\tr(FF^T)u\phi) \ud x=\int_{\partial \Omega}\tr (FF^T)\phi u\cdot n\ud \Gamma=0$ and we obtain}
&=-\frac{\lambda_e}{2} \int_\Omega \frac{\ud}{\ud t} (\tr(FF^T)(1-\phi)) \ud x.
\end{align*}
Thus we have,
	\begin{align}
	\frac12\frac{\ud}{\ud t}\int_\Omega \Big(|u|^2+{\lambda}|\nabla \phi|^2 +\lambda_e |F|^2 + 2\lambda \gamma f(\phi)\Big)\ud x
	&+ \int_\Omega({\eta(\phi)}|\nabla u|^2+\tau |\nabla\mu|^2+\frac{\eta(\phi)}{\kappa(\phi)}u^2)\ud x \nonumber\\ 
	&=\frac{\lambda_e}{2} \frac{\ud }{\ud t}\int_\Omega \tr(FF^T)\phi\ud x 
+\int_\Omega \frac{\eta(\phi)}{\kappa(\phi)}\phi |u|^2 \ud x \label{ener_eq}
\shortintertext{where the right hand side terms can be bounded from above by}
&\lesssim |F||F_t||\phi|_{2} + |F|^2_2|\phi_t|+ |\phi|_2|u|^2.\label{ener_ineq}
\end{align}
Observe that \eqref{ener_eq}-\eqref{ener_ineq} do not provide a closed estimate. We will now derive further higher-order estimates. For that purpose we introduce
\begin{align}\label{defZ}\mathcal{Z}:=|\nabla u|^2+|u_t|^2+| \Delta\phi|^2+|\nabla \phi_t|^2 + |F|^2_2 + |F_t|^2.\end{align}
 We begin by testing \eqref{pde}$_1$ with $A u$ which gives us,
 	\begin{align*}
 	\frac12\frac{\ud}{\ud t}\int_\Omega|\nabla u|^2\ud x +(u\cdot \nabla u, Au)- & (\text{div}(\eta(\phi)\nabla u),Au)=(\lambda\Delta\phi\nabla\phi,Au) \\
 	&	\nonumber-\lambda_e(\text{div}((1-\phi)(FF^T-I)),Au) - (\frac{\eta(\phi)}{\kappa(\phi)}(1-\phi)u,Au) .
 	\end{align*}
 Note that here we have used the relation \eqref{mutophi} to obtain the first term on the right hand side.
 
  We use the classical Stokes theory that suggests the existence of $p^* \in L^2(0,T;H^1(\Omega))$ with $\langle p^*\rangle_{\Omega}=0$, such that $(u,p^*)$ satisfies the equation $-\Delta u+ \nabla p^*=Au$ for a.e $t$. Furthermore, we know that for some $C>0$,
 $$|p^*|_1 \leq C|Au|.$$
 Also observe that, thanks to Proposition 1.2 in \cite{T_NSE}, we have
  $$|p^*|_{L^2(\Omega)/\R} \leq |\nabla p|_{\bH^{-1}(\Omega)} \leq |Au|_{\bH^{-1}(\Omega)} \leq |u|_1.$$ 
 Now we use the above two inequalities and obtain the following estimates.
\begin{align*}
-(\text{div}(\eta(\phi)\nabla u),A u) &= -(\eta(\phi)\Delta u,A u)-(\eta'(\phi)\nabla \phi\cdot \nabla u,A u)\\
& = (\eta(\phi)Au,Au) -(\eta(\phi)\nabla p^*,Au)- (\eta'(\phi)\nabla \phi\cdot \nabla u,A u)\\
& \geq \alpha|Au|^2 + (\eta'(\phi)\nabla \phi\, p^*,Au)-(\eta'(\phi)\nabla \phi\cdot \nabla u,A u).
\end{align*}

The two terms on the right hand side of the above inequality can be treated using the H\"older inequality with powers 2,3,6 and the Sobolev embeddings $H^{\frac12}(\Omega)\subset L^3(\Omega), H^1(\Omega)\subset L^6(\Omega)$ \begin{align*}
|(\eta'(\phi)\nabla\phi \, p^*,Au)|
&\leq C|Au||p^*|^{\frac12}| p^*|_1^{\frac12}|\phi|_2\\
&\leq C|Au|^{\frac32} |u|_1^{\frac12}|\phi|_2\\
& \leq \frac{\alpha}{12}|Au|^2 + C|u|_1^2|\phi|_2^4,\\
| (\eta'(\phi)\nabla \phi\cdot \nabla u,A u)| &\leq C|\nabla \phi|_1|\nabla u|^{\frac12}|\partial \nabla u|^{\frac12} |A u|\\
&\leq \frac{\alpha}{12}|A u|^2 + C|\phi|^4_2|\nabla u|^2.
\end{align*}
Also observe that,
\begin{align*}
|(u\cdot \nabla u,Au)| &\leq |u|_{\bL^6(\Omega)}|\nabla u|_{\bL^3(\Omega)}|Au|\\
&\leq |u|_1|u|_1^{\frac12}|Au|^{\frac12}|Au|\\
&\leq \frac{\alpha}{12}|Au|^2 + C|u|_1^6.
\end{align*}
{We again use interpolation inequalities to obtain the following bounds on the rest of the terms. First,}\begin{align*}
|(\Delta\phi \nabla \phi,A u)| &\leq C |\Delta\phi|_{L^3}|\nabla \phi|_{\bL^6}|Au|\\
&\leq \frac{\alpha}{12}|A u|^2 +C|\Delta\phi|^{}|\nabla\Delta\phi|^{}|\nabla\phi|^2_1 \\&\leq \frac{\alpha}{12}|A u|^2 + {{}\frac{\tau\lambda}{8\tilde{C}}|\phi|_{H^4(\Omega)/ \R}^2} + C| \phi|_2^6\\
& \leq  \frac{\alpha}{12}|A u|^2 +\frac{\tau\lambda}{8}|\Delta^2\phi|^2 + C|\partial_{\bf n}\Delta\phi|^2_{\bH^{\frac12}(\partial\Omega)} + C| \phi|_2^6.
\end{align*}
Here $\tilde{C}>0$ is the same as in Lemma \ref{lem_bi}. Now using \eqref{pde}$_5$ and the boundary conditions \eqref{bc}, we know that
\begin{align*}
\partial_{\bf n}\mu = -\lambda \partial_{\bf n} \Delta \phi + \lambda \gamma f''(\phi) \partial_{\bf n} \phi -\frac{\lambda_e}{2} \partial_{\bf n} \tr(FF^T),
\end{align*}
and hence,
\begin{align*}
\partial_{\bf n}\Delta\phi=-\frac{\lambda_e}{2\lambda}\partial_{\bf n}\tr(FF^T).
\end{align*}
Thanks to an application of a general trace theorem in \cite{LM72}, we know that,
\begin{align}\label{trace}
|\partial_{\bf n}\Delta\phi|_{\bH^{\frac12}(\partial\Omega)} =\frac{\lambda_e}{2} |\partial_{\bf n}\tr(FF^T)|_{\bH^{\frac12}(\partial\Omega)} \leq C|FF^T|_2 \leq C|F|^2_2. 
\end{align}
Thus combining the above arguments we obtain that,
\begin{align*}
|(\Delta\phi\nabla\phi,Au)|\leq  \frac{\alpha}{12}|A u|^2 +\frac{\tau\lambda}{8}|\Delta^2\phi|^2 + C|F|^4_{2} + C| \phi|_2^6.
\end{align*}
Next we have, using that $H^2(\Omega) \subset L^\infty(\Omega)$,
\begin{align*}
|(\nabla \cdot ((1-\phi)(FF^T-I)),A u)| &\leq |(\nabla \phi FF^T,A u)|+ |((1-\phi) \text{ div}(FF^T),A u)|\\
&\leq C |\phi|_{2}|F|^2_2|A u|\\
& \leq \frac{\alpha}{12}|A u|^2 + C|\phi|_2^2|F|^4_2.
\end{align*}
Similarly,
\begin{align*}
\left|(\frac{\eta(\phi)}{\kappa(\phi)}(1-\phi)u,Au)\right|&\lesssim \frac{\beta}{\alpha}|Au||u||\phi|_2\\
&\leq \frac{\alpha}{12}|Au|^2 + C|u|^2|\phi|_2^2.
\end{align*}
All the above estimates together with \eqref{equivnorm} thus give us,
\begin{align*}
\frac12\frac{\ud}{\ud t}|\nabla u|^2 + \frac{\alpha}{2}|A u|^2 \leq \frac{\tau\lambda}{8}|\Delta^2\phi|^2 + C\Big(1+| u|_1^6+ |\Delta \phi|^8+|\Delta\phi|^6 + |\Delta\phi|^4|u|_1^2 +|\Delta\phi|^2|F|^4_2+ |F|^4_2\Big).
\end{align*}
That is, with $\mathcal{Z}$ as defined in \eqref{defZ}, we obtain,
\begin{align}\label{uH1}
\frac12\frac{\ud}{\ud t}|\nabla u|^2 + \frac{\alpha}{2}|A u|^2 \leq \frac{\tau\lambda}{8}|\Delta^2\phi|^2 + C(1+\mathcal{Z})^4.
\end{align}
Next, we test \eqref{pde}$_4$ with 
$\Delta^2 \phi$. 
Using \eqref{tensor} and \eqref{pde}$_5$ we obtain,
\begin{align*}
\frac12\frac{\ud }{\ud t}|\Delta \phi|^2 + {\int_{\partial \Omega}\phi_t\partial_{\bf n}\Delta \phi\ud \Gamma}+(u \cdot \nabla \phi,\Delta^2 \phi) +& \tau\lambda|\Delta^2 \phi|^2=- (\frac{\lambda_e}{2}\tau\Delta \tr(FF^T),\Delta^2\phi) \\&+\tau\lambda\gamma(f''(\phi)\Delta\phi+f'''(\phi)|\nabla \phi|^2,\Delta^2\phi) .
\end{align*}
We first treat the boundary term in the above equation using the boundary conditions \eqref{bc}. Observe that since $\partial_{\bf n}\phi_t=0$ on $\partial\Omega$, thanks to the divergence theorem, we have
\begin{align*}
|\int_{\partial \Omega}\phi_t\partial_{\bf n}\Delta \phi\ud \Gamma| &= |\int_{\partial \Omega}\partial_{\bf n}(\phi_t\Delta\phi )\ud \Gamma|\\
&=|\int_\Omega \Delta (\phi_t \Delta \phi) \ud x|\\
&=|\int_\Omega \Delta\phi_t \Delta \phi + 2\nabla \phi_t \cdot\nabla\Delta \phi
+\phi_t \Delta^2 \phi \ud x|\\
&=|\int_\Omega  \nabla\Delta\phi_t \cdot \nabla \phi
+\phi_t \Delta^2 \phi \ud x|\\
& \leq C |\nabla \Delta\phi_t||\nabla\phi| + |\phi_t||\Delta^2\phi|\\
&\leq \frac{\tau\lambda}{8}|\nabla\Delta\phi_t|^2 + \frac{\tau\lambda}{8}|\Delta^2\phi|^2 + C|\nabla\phi|^2 + C|\phi_t|^2.
\end{align*}

\noindent Next, using again the H\"older and Sobolev inequalities we find the following estimates:
\begin{align*}
|(u \cdot \nabla \phi,\Delta^2 \phi)| &\leq \frac{\tau\lambda}{8}|\Delta^2\phi|^2 + C|u|^2_1|\phi|^2_2,\\
|(\Delta \tr(FF^T),\Delta^2\phi)| &\leq |\Delta^2\phi||\Delta \tr(FF^T)|\\
&\leq |\Delta^2\phi||F|_2|F|_{\bL^\infty(\Omega)}\\ 
&\leq \frac{\tau\lambda}{8}|\Delta^2\phi|^2 + C|F|^4_2.
\end{align*}
Recalling that $f''$ and $f'''$ are quadratic and linear respectively, we also obtain
\begin{align*}
|(f''(\phi)\Delta\phi+f'''(\phi)|\nabla \phi|^2,\Delta^2\phi)|&\lesssim |\phi|_{L^\infty(\Omega)}^{2}|\Delta\phi||\Delta^2\phi| + |\phi|_{L^\infty(\Omega)}|\nabla \phi|_{\bL^4(\Omega)}^2|\Delta^2\phi|\\
&\lesssim |\phi|_2^{2}|\Delta\phi||\Delta^2\phi| + |\phi|_{2}|\nabla \phi|_1^2|\Delta^2\phi|\\
&\leq \frac{\tau\lambda}{8}|\Delta^2\phi|^2 + C|\phi|^6_2.
\end{align*}
Hence, recalling also \eqref{equivnorm}, we obtain,
\begin{align*}
\frac12\frac{\ud }{\ud t}| \Delta\phi|^2  + \frac{\tau\lambda}{2}|\Delta^2 \phi|^2
\leq \frac{\tau\lambda}{8}|\nabla\Delta\phi_t|^2   +  C(1+|\Delta\phi|^6 + |u|^2_1|\Delta\phi|^2+|\nabla\phi|^2+|\phi_t|^2+ |F|^4_2) .
\end{align*}
That is we have,
\begin{align}\label{phiH2}
\frac12\frac{\ud }{\ud t}| \Delta\phi|^2  + \frac{\tau\lambda}{2}|\Delta^2 \phi|^2
\leq \frac{\tau\lambda}{8}|\nabla\Delta\phi_t|^2   +  C(1+\mathcal{Z})^3 .
\end{align}
Now we apply $\partial_t$ to \eqref{pde}$_1$, use \eqref{mutophi} and then test with $u_t$: we find
\begin{align*}
(u_{tt},u_t) + (u_t \cdot \nabla u,u_t)+ (\eta'(\phi)\phi_t \nabla u,\nabla u_t) + \int_\Omega\eta(\phi)|\nabla u_t|^2\ud x=(\Delta \phi_t\nabla \phi+\Delta\phi\nabla \phi_t,u_t)\\
\qquad+\lambda_e(\phi_t(FF^T-I)+\partial_t(FF^T)(1-\phi),\nabla u_t)\\
-\left( (\frac{\eta'(\phi)}{\kappa(\phi)} -\frac{\eta(\phi)}{\kappa^2(\phi)}\kappa'(\phi))\phi_t(1-\phi)u - \eta(\phi)\phi_t u +\eta(\phi)(1-\phi)u_t,u_t\right).
\end{align*}
Now we will obtain estimates on each term appearing in the above equation. First observe that,
\begin{align*}
|(u_t \cdot \nabla u,u_t)| &\lesssim |u_t|_{\bL^3(\Omega)}|\nabla u||u_t|_{\bL^6(\Omega)} \\
&\lesssim |u_t|^{\frac12}|u_t|_1^{\frac32}|u|_1\\
&\leq \frac{\alpha}{8}|u_t|^2_1 + C|u_t|^2|u|^4_1.\end{align*}
Similarly using interpolation inequalities 
we obtain,
\begin{align*}|(\eta'(\phi)\phi_t \nabla u,\nabla u_t)| & \lesssim |\eta'|_{L^\infty(\Omega)} |\phi_t|_{L^6(\Omega)}|\nabla u|_{\bL^3(\Omega)}|\nabla u_t|\\&\lesssim |\eta'|_{L^\infty(\Omega)} |u_t|_1|\phi_t|_1|u|^{\frac12}_1|u|_2^{\frac12}\\& \leq \frac{\alpha}{8}|u_t|_1^2 + \frac{\alpha}{4}|u|_2^2 + C|\phi_t|_1^4|u|_1^2. \end{align*}
Next,
{\begin{align*}
|(\Delta \phi_t\nabla \phi+\Delta\phi\nabla \phi_t,u_t)| & \lesssim |(\Delta \phi_t,(u_t\cdot \nabla \phi))| +|\Delta\phi\nabla \phi_t,u_t)| \\
&\lesssim |\Delta \phi_t||u_t|_{\bL^3(\Omega)}|\nabla\phi|_{\bL^6(\Omega)}+ |\nabla\phi_t||u_t|_{\bL^6(\Omega)}|\Delta\phi|_{L^3(\Omega)}\\
&\lesssim |\Delta \phi_t||u_t|^{\frac12}|\nabla u_t|^{\frac12}|\phi|_2+ |\nabla\phi_t||u_t|_1|\Delta\phi|^{\frac12}|\Delta\phi|_1^{\frac12};
\shortintertext{applying the Schwarz inequality and the generalized Poincar\'e inequality since  $\langle\Delta\phi_t\rangle_\Omega=0$, we obtain}
&\leq \frac{\tau\lambda}{8}|\nabla\Delta \phi_t|^2 + C|u_t||u_t|_1|\phi|_2^2+ \frac{\tau\lambda}{8\tilde{C}}|\Delta\phi|_1+ C|\nabla\phi_t|^2|u_t|^2_1|\Delta\phi|\\
& \leq \frac{\tau\lambda}{8}|\nabla\Delta \phi_t|^2+\frac{\alpha}{8}|u_t|_1^2 +   C|u_t|^2|\phi|_2^4 + {\frac{\tau\lambda}{8\tilde{C}}|\phi|_{H^4(\Omega)/ \R}^2} +C|\nabla\phi_t|^4|\Delta\phi|^2;
\shortintertext{using Lemma \ref{lem_bi} and \eqref{trace} again, we further obtain}
&\leq \frac{\tau\lambda}{8}|\nabla\Delta \phi_t|^2+\frac{\alpha}{8}|u_t|_1^2 + {\frac{\tau\lambda}{8\tilde{C}}|\phi|_{H^4(\Omega)/ \R}^2}+C|u_t|^2|\phi|_2^4  + C|\nabla\phi_t|^4|\Delta\phi|^2\\
&\hspace{-1in}\leq \frac{\tau\lambda}{8}|\nabla\Delta \phi_t|^2+\frac{\alpha}{8}|u_t|_1^2 + {\frac{\tau\lambda}{8}|\Delta^2\phi|^2} + C|\partial_{\bf n}\Delta\phi|^2_{\bH^{\frac12}(\partial\Omega)}+C|u_t|^2|\phi|_2^4 + C|\nabla\phi_t|^4|\Delta\phi|^2\\
&\hspace{-0.3in}\leq \frac{\tau\lambda}{8}|\nabla\Delta \phi_t|^2+\frac{\alpha}{8}|u_t|_1^2 + {\frac{\tau\lambda}{8}|\Delta^2\phi|^2} + C|F|^4_{2}+ C|u_t|^2|\phi|_2^4 +C|\nabla\phi_t|^4|\Delta\phi|^2.
\end{align*}}
\noindent
Similarly using the fact that $H^2(\Omega)\subset L^\infty(\Omega)$ we observe that,\begin{align*}
|(\phi_t(FF^T-I)+\partial_t(FF^T)(1-\phi),\nabla u_t)| &\lesssim |\phi_t||F|^2_{\bL^\infty(\Omega)}|\nabla u_t| + |F|_{\bL^\infty(\Omega)}|F_t|{|1-\phi|_{L^\infty(\Omega)}}|\nabla u_t|\\
&\lesssim |\phi_t||F|^2_2|u_t|_1 + |F|_2|F_t||1-\phi|_2|u_t|_1\\
& \leq \frac{\alpha}{8}|u_t|^2_1 + C|F|^4_2|\phi_t|^2+C|F|^2_2|F_t|^2{(1+|\phi|^2_2)},
\shortintertext{and,}
&\hspace{-2.8in}\left|\left( (\frac{\eta'(\phi)}{\kappa(\phi)} -\frac{\eta(\phi)}{\kappa^2(\phi)}\kappa'(\phi))\phi_t(1-\phi)u - \eta(\phi)\phi_t u +\eta(\phi)(1-\phi)u_t,u_t\right)\right| \\&\lesssim |\phi|_2|\phi_t|_1|u|_1 |u_t|+ |\phi|_2|u_t|^2. 
\end{align*}
Combining all the bounds and recalling \eqref{equivnorm} and \eqref{K0} we obtain,
\begin{equation*}\begin{split}
\frac12\frac{\ud}{\ud t}|u_t|^2 + \frac{\alpha}2 |u_t|^2_1 &\leq \frac{\alpha}{4}|u|_2^2+ \frac{\tau\lambda}{8}|\Delta^2\phi|^2+\frac{\tau\lambda}{8}|\nabla\Delta \phi_t|^2\\ 
& +C\Big(1+|u_t|^2(|\Delta\phi|^4+|u|^4_1)+|\phi_t|_1^4|u|_1^2+|\phi_t|_1^4|\Delta\phi|^2+|F|^4_2(|\phi_t|^2+1)\Big)\\
& +C\Big(|F|^2_2|F_t|^2{(1+|\Delta\phi|^2_2)}+|\Delta\phi||\phi_t|_1|u|_1 |u_t|\Big).
\end{split}\end{equation*}
In other words we have,
\begin{equation}\begin{split}\label{ut}
\frac12\frac{\ud}{\ud t}|u_t|^2 &+ \frac{\alpha}2 |u_t|^2_1 \leq \frac{\alpha}{4}|u|_2^2+ \frac{\tau\lambda}{8}|\Delta^2\phi|^2+\frac{\tau\lambda}{8}|\nabla\Delta \phi_t|^2 +C(1+\mathcal{Z})^3.
\end{split}\end{equation}
Next we apply $\partial_t$ to \eqref{pde}$_4$, test it with $-\Delta \phi_t$ and then use the expression \eqref{pde}$_5$ for $\mu$ to obtain,
\begin{align*}
&\frac12\frac{\ud}{\ud t}|\nabla \phi_t|^2  -(u_t \cdot \nabla \phi + u \cdot \nabla \phi_t,\Delta\phi_t)=\tau(\nabla \mu_t, \nabla \Delta \phi_t)\\
&=-\tau\lambda |\nabla \Delta \phi_t|^2+\frac{\lambda_e}{2}\tau (\partial_t(\nabla\tr(FF^T)),\nabla\Delta \phi_t)
 - \tau\lambda\gamma (f''(\phi)\nabla\phi_t+f'''(\phi)\phi_t\nabla\phi  
 ,\nabla\Delta \phi_t).
\end{align*}
Using \eqref{avs} along with the generalized Poincar\'e inequality, we obtain the following estimates on the terms appearing in the equation above.
\begin{align*}
|(u_t \cdot \nabla \phi + u \cdot \nabla \phi_t,\Delta\phi_t)| &\leq C|u_t||\nabla \phi|_1|\nabla \Delta \phi_t| + C|u|_1|\nabla \phi_t||\nabla \Delta \phi_t|\\
&\leq \frac{\tau\lambda}{6}|\nabla \Delta \phi_t|^2 + C(|u_t|^2|\phi|^2_2 + |u|_1^2|\nabla\phi_t|^2),\end{align*}and,
\begin{align*}
|(\partial_t(\nabla\tr(FF^T)),\nabla\Delta \phi_t)| &\leq \left(|F_t|_1|\nabla F|_1+ |F|_{\bL^\infty(\Omega)}|\nabla F_t|\right)|\nabla \Delta\phi_t|
\\& \leq \frac{\tau\lambda}{6}|\nabla \Delta \phi_t|^2 + C|F_t|_1^2|F|_2^2.
\end{align*} Next we use the fact that $f$ is a polynomial of degree $4$ to observe,
\begin{align*}
|(f''(\phi)\nabla\phi_t+f'''(\phi)\phi_t\nabla\phi ,\nabla\Delta \phi_t)| & \lesssim |\phi|^2_{L^\infty(\Omega)}|\nabla \phi_t||\nabla \Delta\phi_t| \\
& \hspace{1in}+ |\phi|_{L^\infty(\Omega)}| \phi_t|_{L^6(\Omega)}|\nabla \phi|_{\bL^3(\Omega)}|\nabla \Delta \phi_t|\\
&\lesssim|\nabla \Delta \phi_t||\nabla \phi_t||\phi|^2_{2} \\
&  \leq \frac{\tau\lambda}{6}|\nabla \Delta \phi_t|^2 + C|\nabla \phi_t|^2|\phi|^4_{2}.
\end{align*} 
Hence, combining the above estimates and using \eqref{equivnorm} we obtain,
\begin{equation}\begin{split}\label{phitH1_0}
\frac12\frac{\ud}{\ud t}|\nabla \phi_t|^2 &+\frac{\tau\lambda}{2} |\nabla \Delta \phi_t|^2 \leq C(1+|u_t|^2|\Delta\phi|^2 + |u|_1^2|\nabla\phi_t|^2+|F_t|_1^2|F|_2^2+|\nabla \phi_t|^2|\Delta\phi|^4).
\end{split}\end{equation}
Observe also that \eqref{pde}$_3$ gives us,
\begin{equation}\begin{split}
|F_t|_1 &\leq |u \cdot \nabla F|_1 + |\nabla u F|_1\\
&\leq C|u|_2|F|_2 \\
& \leq C|u|_1^{\frac12}|u|_3^{\frac12}|F|_2.\label{FtH1}
\end{split}\end{equation}
Substituting \eqref{FtH1} in \eqref{phitH1_0} and thanks to \eqref{equivnorm} we obtain,
\begin{equation}\begin{split}\label{phitH1}
\frac12\frac{\ud}{\ud t}|\nabla \phi_t|^2 &+\frac{\tau\lambda}{2} |\nabla \Delta \phi_t|^2 \leq C|u|_1|u|_3|F|^4_2+ C(1+\mathcal{Z})^3.
\end{split}\end{equation}
We next apply 
$\Delta$ to \eqref{pde}$_3$ and then test with $\Delta F$. That is we consider, 
$$\int_\Omega\tr(\Delta F_t \Delta F^T) \ud x +\int_\Omega \tr(\Delta (u\cdot \nabla F) \Delta F^T)\ud x=\int_\Omega\tr(\Delta(\nabla uF)\Delta F^T)\ud x.$$
Thanks to the divergence free property of $u$, we obtain
\begin{align*}
|\int_\Omega \tr(\Delta(u\cdot \nabla F)\Delta F^T)\ud x|&=|\int_\Omega2\sum_{i,j=1}^d\tr(\nabla u\nabla^2F^{ij})\Delta F^{ij}+\tr(\Delta u \cdot\nabla F \Delta F^T)\ud x| \\
&\leq C |u|_3|F|^2_2,
\end{align*}
where $\nabla^2$ denotes the Hessian. \\
Next observe that,
\begin{align*}
|\int_\Omega \tr(\Delta (\nabla u F)\Delta F^T )\ud x| &\leq C |\nabla u F|_2|F|_2 \\
&\leq C|u|_3|F|_{\bL^\infty(\Omega)}|F|_2\\
&\leq C|u|_3|F|^2_2.
\end{align*}
Since $\int_\Omega\tr(\Delta F_t \Delta F^T) \ud x=\frac12\frac{d}{dt}|\Delta F|^2$, we thus have,
\begin{align}\label{FH2}
\frac12 \frac{\ud }{\ud t}|\Delta F|^2 & \leq C|u|_3|F|^2_2.
\end{align}
Next we apply $\partial_t$ to \eqref{pde}$_3$ and test {the resulting equation} with $F_t$:
$$\int_\Omega\tr( F_{tt}  F^T_t) \ud x +\int_\Omega \tr( (u\cdot \nabla F)_t  F_t^T)\ud x=\int_\Omega\tr((\nabla uF)_t F_t^T)\ud x.$$
Observe that $\int_\Omega\tr((u \cdot \nabla F_t) F_t^T)\ud x=0$ and,
\begin{align*}
|\int_\Omega\tr(u_t \cdot \nabla F F_t^T)\ud x| &\leq C|u_t|_1|F|_2|F_t|\\
&\leq \frac{\alpha}{8}|u_t|_1^2 + C|F|^2_2|F_t|^2.
\end{align*}
Similarly we also have,
\begin{align*}
|\int_\Omega \tr[(\nabla u_t F + \nabla u F_t)F_t^T]\ud x| &\lesssim |\nabla u_t||F|_2|F_t| + |u|_3|F_t|^2\\
&\leq \frac{\alpha}{8}|u_t|_1^2 + C|F|^2_2|F_t|^2 + C|u|_3|F_t|^2.
\end{align*}
Thus we obtain
\begin{align*}
\frac12 \frac{\ud}{\ud t}|F_t|^2 \leq \frac{\alpha}{4}|u_t|_1^2 + C\left( |F|^2_2|F_t|^2 + |u|_3|F_t|^2\right) . 
\end{align*}
That is,
\begin{align}\label{Ft}
\frac12 \frac{\ud}{\ud t}|F_t|^2 \leq \frac{\alpha}{4}|u_t|_1^2 + C|u|_3|F_t|^2+ C(1+\mathcal{Z})^2. 
\end{align}
Now we combine \eqref{uH1}, \eqref{phiH2}, \eqref{ut}, \eqref{phitH1}, \eqref{FH2} and \eqref{Ft} to obtain,
\begin{equation}\begin{split}\label{ineq1}
&\frac12 \frac{\ud}{\ud t}\Big(|\nabla u|^2+|u_t|^2+| \Delta\phi|^2+|\nabla \phi_t|^2 + |F|^2_2 + |F_t|^2\Big)\\
&\qquad\quad + \frac{\alpha}{4}| u|_2^2+ \frac{\alpha}4 |u_t|^2_1 +  \frac{\tau\lambda}{4}| \Delta^2\phi|^2+\frac{\tau\lambda}{4} |\nabla \Delta \phi_t|^2 
\\
&\hspace{1in} 
\leq C\Big(1+| u|_1^6+ |\phi|^8_2+|\phi|^6_2+|F|^4_2 + |\phi|^4_2|u|_1^2 +|\phi|^2_2|F|^4_2+|u|^2_1|\Delta\phi|^2\Big)\\
&\hspace{1in}  +C\Big(|u_t|^2|u|^4_1+|\phi_t|_1^4|u|_1^2+|\phi_t|_1^4|\Delta\phi|^2+|F|^4_2|\phi_t|^2+|F|^2_2|F_t|^2|\Delta\phi|^2\Big)\\
&\hspace{1in}  + C\left( |u_t|^2|\Delta\phi|^2 + |u|_1^2|\phi_t|_1^2+| \phi_t|_1^2|\Delta\phi|^4+|F|^2_2|F_t|^2\right) \\
&\hspace{1in}  +C|u|_3\left( |u|_1|F|_2^4+|F|^2_2 + |F_t|^2\right) .
\end{split}\end{equation}
That is, for $\mathcal{Z}$ as defined in \eqref{defZ}, we have
\begin{equation}\begin{split}\label{ineqZ1}
\frac12 \frac{\ud}{\ud t}\mathcal{Z} +  \mathcal{M}
 \leq C_0|u|_3 \mathcal{Z}^3+C(1+\mathcal{Z})^4 .
\end{split}\end{equation}
where $C_0>0$ and,
\begin{align}
\mathcal{M}&:=\frac{\alpha}{4}| u|_2^2+ \frac{\alpha}4 |u_t|^2_1 +  \frac{\tau\lambda}{4}| \Delta^2\phi|^2+\frac{\tau\lambda}{4} |\nabla \Delta \phi_t|^2.\label{defM}\end{align}
Hence to close our estimates we need to control $|u|_3$ appearing on the right hand side of \eqref{ineqZ1}.\\
For that purpose we use Lemma \ref{Stokes} on equation \eqref{pde}$_1$ and interpolation inequalities as follows,
\begin{equation}\begin{split}\label{app_stokes}	|u|_3 + \left|\frac{p}{\eta(\phi)}\right|_1 &\leq C \left[ 1+ \left( 1+ | \phi|_{2}^2\right) \left( |\phi|_2^2+ |\phi|_2^{\frac12}|\phi|_3^{\frac12}\right)  \right] \Big( |u_t|_1 + |u\cdot \nabla u|_1 \\&+ |\nabla \cdot (1-\phi)(FF^T-I)|_1 +| \Delta\phi\nabla \phi|_1+\left|\eta(\phi)\frac{(1-\phi)u}{\kappa(\phi)}\right|_1\Big).
\end{split}\end{equation}
The above form is obtained by using equation \eqref{tensor} in \eqref{pde}$_1$ and including the gradient term on the right hand side of \eqref{tensor} in the pressure gradient appearing in \eqref{pde}$_1$.\\
Now we treat each one of the terms appearing on the right hand side of the above inequality using interpolation inequalities and the generalized Poincar\'e inequality.
\begin{align*}
|u\cdot \nabla u|_1 &\leq C |\nabla (u\cdot \nabla u)| \\
&\leq C(|\nabla u \cdot \nabla u| + |u\cdot \nabla \partial u|)\\
&\leq C|\nabla u|_{\bL^3(\Omega)}|\nabla u|_{\bL^6(\Omega)}+ |u|_{\bL^6(\Omega)}|Du|_{\bL^3(\Omega)}\\
& \leq C\left( |\nabla u|^{\frac12}|u|_2^{\frac32} + |\nabla u||u|_2^{\frac12}|u|_3^{\frac12}\right) .\end{align*}
Next, observe that the continuous embedding $H^2(\Omega)\subset L^\infty(\Omega)$ gives us,
\begin{align*}
|\text{div}(1-\phi)(FF^T-I)|_1 &\leq |\nabla \phi(FF^T-I)|_1 + |(1-\phi) \nabla FF^T|_1\leq C|\phi|_2|F|^2_2;\end{align*}
{additionally we have,}\begin{align*}
|\Delta\phi \nabla \phi|_1 &\leq C|\Delta \phi|_1|\nabla \phi|_{\bL^\infty(\Omega)} \\
&\leq C|\phi|_3^2\\
& \leq C|\phi|_2|\phi|_4,
\shortintertext{and similarly,} 
\left|\eta(\phi)\frac{(1-\phi)u}{\kappa(\phi)}\right|_1 & \leq C(\beta) |\phi|_2|u|_1.
\end{align*}
Combining all the bounds above and using the interpolation inequality $|\phi|_3\leq C|\phi|_2^{\frac12}|\phi|^{\frac12}_4$, we obtain
\begin{align}
|u|_3  &\leq C \left[ \left( 1+ | \phi|_{2}^2\right) \left( |\phi|^2_2+|\phi|_2^{\frac34}|\phi|_4^{\frac14}\right)  \right] \Big( |u_t|_1+  |\nabla u|^{\frac12}|u|_2^{\frac32} 
+|\phi|_2|F|^2_2+|\phi|_2|\phi|_4
+|\phi|_2|u|_1\Big) \nonumber\\
& +C\left[ \left( 1+ | \phi|_{2}^4\right) \left( |\phi|^4_2+ |\phi|_2^{\frac32}|\phi|_4^{\frac12}\right)  \right]|\nabla u|^2|u|_2.\label{uH3}
\end{align}

Next we aim to use \eqref{uH3} in \eqref{ineqZ1}. First we treat each term appearing in the bounds for $|u|_3\mathcal{Z}^3$, obtained using \eqref{uH3}, by applying the H\"older inequality as follows:
\begin{align*}
C_0|u|_3\mathcal{Z}^3 &\lesssim \left[\left( 1+ | \phi|_{2}^2\right) \left( |\phi|^2_2+ |\phi|_2^{\frac34}|\phi|_4^{\frac14}\right)  \right] \Big(  |u_t|_1+ | u|_1^{\frac12}|u|_2^{\frac32} 
+|\phi|_2(|F|^2_2+|u|_1)+|\phi|_2|\phi|_4\Big)\mathcal{Z}^3\\
&\qquad +\left[ \left( 1+ | \phi|_{2}^4\right) \left( |\phi|^4_2+ |\phi|_2^{\frac32}|\phi|_4^{\frac12}\right)  \right]| u|_1^2|u|_2\mathcal{Z}^6 \\
& \leq \frac{\alpha}{8} |u|_2^{2} + \frac{\alpha}{8}|u_t|_1^2+ \frac{\tau\lambda}{8\tilde{C}}|\phi|_4^2 \\
&\qquad +C\left(|u|_1^4(1+|\phi|_2^{2})|\phi|_2^{\frac34}\mathcal{Z}^{3}\right)^8+C\left((|F|^2_2+|u|_1)|\phi|_2^{\frac74}(1+|\phi|_2^{2})\mathcal{Z}^3\right)^{\frac87}\\
& \qquad +C\left((1+|\phi|_2^{2})|\phi|_2^{\frac74}\mathcal{Z}^3\right)^{\frac83}+ C \left( |u|_1^2(1+|\phi|_2^{4})|\phi|_2^{\frac32}\mathcal{Z}^{6}\right)^4 + \,\, \text{l.o.t.} \end{align*}
That is we have,
\begin{align}\label{u3z}
C_0|u|_3\mathcal{Z}^3 &\leq  \frac{\alpha}{8} |u|_2^{2} + \frac{\alpha}{8}|u_t|_1^2+ \frac{\tau\lambda}{8\tilde{C}}|\phi|_4^2 + C(1+\mathcal{Z})^{27}.
\end{align}
We again use Lemma \ref{lem_bi} and \eqref{trace} to write
$$\frac{\tau\lambda}{8\tilde{C}}|\phi|^2_4 \leq \frac{\tau\lambda}{8}|\Delta^2\phi|^2 + C(1+|F|^4_2).$$
We now substitute \eqref{u3z} in \eqref{ineqZ1}. Recalling the definition of $\mathcal{Z}$ and $\mathcal{M}$ given in \eqref{defZ} and \eqref{defM} respectively, we write
\begin{equation}\begin{split}\label{ineq}
& \frac{\ud}{\ud t}\mathcal{Z} +  \mathcal{M}
\leq  C_1 (1+\mathcal{Z})^{27},
\end{split}\end{equation}
where $C_1$ is an absolute constant depending on the data. We infer from \eqref{ineq} that
$$1+\mathcal{Z}(t) \leq 1+\mathcal{Z}(0) +C_12^{27}(1+\mathcal{Z}(0))^{27}t$$
as long as $1+\mathcal{Z}(t) \leq 2(1+\mathcal{Z}(0))$, which happens for $t \leq T_0$,
$$1+\mathcal{Z}(0)+C_12^{27}(1+\mathcal{Z}(0))^{27}T_0 \leq 2(1+\mathcal{Z}(0))$$
that is,
\begin{align}
t \leq T_0:=\frac{1}{C_12^{26}(1+\mathcal{Z}(0))^{26}}.
\end{align}
We now infer from \eqref{ineq} that for $0 \leq t \leq T_0$,
\begin{align}
1+\mathcal{Z}(t) &\leq 2(1+\mathcal{Z}(0)),\label{boundZ}\\
\text{and, }\int_0^{T_0}\mathcal{M}(s)\ud s &\leq  2(1+\mathcal{Z}(0)).\label{boundM}
\end{align}
We give below the interpretation of \eqref{boundZ}-\eqref{boundM} in terms of the functions, but we first observe that the calculations are complete (i.e. closed), if we express $\mathcal{Z}(0)$ in terms of the initial data, that is $u_t(0),\phi_t(0),F_t(0)$, which we do now.
\\
Thanks to \eqref{mutophi}, we have,
\begin{equation}\begin{split}\label{ut0}
|u_t(0)| &\leq |u_0\cdot \nabla u_0| + | \eta'(\phi_0)\nabla \phi_0\cdot \nabla u_0| +| \eta(\phi_0)\Delta u_0|+ |\nabla\phi_0(F_0F_0^T -I)|\\
&\qquad \quad+  |(1-\phi_0)\text{div}(F_0F_0^T )|  + 
|\Delta\phi_0\nabla \phi_0|+ |\frac{\eta(\phi_0)}{\kappa(\phi_0)}(1-\phi_0)u_0|\\
& \lesssim |u_0|_2^2 + |\phi_0|_2|u_0|_2 + |u_0|_2 + |\phi_0|_2|F_0|_2^2 + |\phi_0|_3|\phi_0|_2.
\end{split}
\end{equation}
Using the fact that $f''$ is quadratic we see,
\begin{equation}\label{phit0}
\begin{split}
|\phi_t(0)|_1 &\leq |(u_0\cdot \nabla\phi_0)|_1 + |\phi_0|_3 + |f''(\phi_0)\nabla \phi_0| + |\tr(F_0F_0^T)|_1\\
&\leq |u_0|_2|\phi_0|_2 + |\phi_0|_3 + |\phi_0|_2^3 +|F_0|^2_2.
\end{split}
\end{equation}
Similarly,
\begin{equation}\label{Ft0}
\begin{split}
|F_t(0)| &\leq |u_0\cdot \nabla F_0| + |\nabla u_0 F_0|\\
& \leq |u_0|_2|F_0|_2.
\end{split}\end{equation}
Thanks to the Gr\"onwall inequality and \eqref{ut0},\eqref{phit0} and \eqref{Ft0}, we know that there exist constants $T_0>0$ and $K_1>0$ depending only on the initial conditions such that,
\begin{align}\label{linfty_est}
 |u(t)|_1^2+|u_t(t)|^2+|\Delta \phi(t)|^2+| \phi_t(t)|_1^2 + |F(t)|^2_2 + |F_t(t)|^2 \leq K_1 \qquad \forall t \in [0,T_0].
\end{align}
Now we integrate \eqref{ineq} over $[0,T_0]$ and use \eqref{linfty_est} to obtain that there exists $K_2>0$ that depends on $T_0$ and $K_1$ such that,
\begin{equation}\label{l2_est}
\int_0^{T_0}(\frac{\alpha}{8}|\Delta u|^2+ \frac{\alpha}4 |u_t|^2_1 +  \frac{\tau\lambda}{8}| \Delta^2\phi|^2+\frac{\tau\lambda}{2} |\nabla \Delta \phi_t|^2 )\ud t \leq K_2.
\end{equation}
Furthermore, using Lemma \ref{Stokes} and earlier estimates on the right hand side terms in \eqref{app_stokes}, we have,
\begin{align*}	
|u|_3  &\leq C \left[ \left( 1+ | \phi|_{2}^2\right) \left( |\phi|_2^{\frac12}|\phi|_3^{\frac12}\right)  \right] \Big( |u_t|_1+  |\nabla u|^{\frac12}|u|_2^{\frac32} 
+|\phi|_2|F|^2_2+|\phi|_2(|\Delta^2\phi|+|F|_2^2)
+|\phi|_2|u|_1\Big)\\
& +C\left[ \left( 1+ | \phi|_{2}^4\right) \left(  |\phi|_2^{}|\phi|_3^{}\right)  \right]|\nabla u|^2|u|_2,\end{align*}
which gives, for some $K_3>0$,
\begin{align}\label{ul2h3}
|u|_{L^2(0,T_0;\bH^3(\Omega))} \leq K_3.
\end{align}
{Now we sum up our findings in \eqref{linfty_est}, \eqref{l2_est} and \eqref{ul2h3} and conclude that for some $K_4>0$ and any $t \in [0,T_0]$,
		\begin{equation}\begin{split} \label{ests}
	|u(t)|_1^2+|u_t(t)|^2+|\Delta \phi(t)|^2+| \phi_t(t)|_1^2 + |F(t)|^2_2 + |F_t(t)|^2\\
+\int_0^{T_0}(| u|_3^2+  |u_t|^2_1 +  | \Delta^2\phi|^2+|\nabla \Delta \phi_t|^2 )\ud t	\leq K_4.
\end{split}\end{equation}
}

With these final estimates we conclude this section.

\section{Galerkin scheme}\label{galerkin}
\begin{proof}[Proof of Theorem \ref{main}] We will employ the Faedo-Galerkin approximation method in this proof.
	As usual, to construct the Galerkin scheme for equation \eqref{pde}-\eqref{bc} we use the orthonormal basis {$(w_k)^\infty_{k=1}$} and {$(e_k)^\infty_{k=1}$} of $H$ and $L^2(\Omega)^d$ respectively consisting of the eigenvectors of the Stokes operator $A$ and that corresponding to the Neumann eigenvalues of the operator $-\Delta +I$ respectively. Similarly, for the $F$ equation we consider eigenfunctions $(M_n)_{n=1}^\infty$ of the Laplace operator that form
	an orthonormal basis of $L^2(\Omega)^{d\times d}$. For each $n\ge 1$ consider the $n$-dimensional subspaces of $H$, $L^2(\Omega)^d$ and $L^2(\Omega)^{d\times d}$:
	\begin{align*}
	V^1_n:
	&	=\text{span}\{w_1,...,w_n\}, \\
	V^2_n:
	&	=\text{span}\{e_1,...,e_n\},\\
	V^3_n:
	&	=\text{span}\{M_1,...,M_n\}.
	\end{align*}
	Let $\Pj^1_n: H\rightarrow V^1_n \,$, $\,\Pj^2_n: L^2(\Omega)\rightarrow V^2_n$ and $\Pj^3_n: L^2(\Omega)^{d\times d}\rightarrow V^3_n$ be the orthonormal projections.\\
	
	Now for $v \in V^1_n$, $\psi \in V^2_n$ and $\Xi \in V^3_n$ we consider the approximating equations
	\begin{equation}\begin{split}\label{galerkin_eqn}
	&\langle\frac{\partial u_n}{\partial t},v\rangle + (u_n \cdot \nabla u_n,v) +  (\eta(\phi_n)\nabla u_n,\nabla v) =  (\lambda_e(1-\phi_n)(F_nF_n^T-I),\nabla v) \\&\qquad \qquad +\lambda  (\nabla \phi_n \otimes \nabla \phi_n,\nabla v)
	+(\eta(\phi_n)\frac{(1-\phi_n)u_n}{\kappa(\phi_n)},v),\\
	&\langle\frac{\partial \phi_n}{\partial t},\psi\rangle + (u_n \cdot \nabla \phi_n,\psi) + \tau(\nabla \mu_n, \nabla \psi)=0,\\
	&\langle\frac{\partial F_n}{\partial t}, \Xi\rangle+ (u_n\cdot \nabla F_n,\Xi)=(\nabla u_n F_n,\Xi).
\end{split}	\end{equation} 
	Here we have, 
	\begin{align}
	\mu_n:=\Pj^2_n\big(-\lambda \Delta \phi_n  + \lambda\gamma f'(\phi_n)&-\frac{\lambda_e}{2}\tr(F_nF_n^T-I)\big).
	\end{align}
	Using standard techniques, one can prove the existence of $u_n, \phi_n$ and $F_n$ that solve \eqref{galerkin_eqn} for some time $0<t_n \leq T$. The a priori estimates derived in \eqref{ests} carry over to $u_n,\phi_n,F_n$ {\it with the same defintion of $T_0$ as in \eqref{ests}}. Hence, we see that $t_n \geq T_0$ and,
\begin{align*}
&u_n \text{ is bounded in } L^\infty(0,T_0;V) \cap L^2(0,T_0;H^3(\Omega)^d) \text{ independently of } n,\\ &\phi_n \text{ is bounded in } L^\infty(0,T_0;H^2(\Omega)) \cap L^2(0,T_0;H^4(\Omega)) \text{ independently of } n,\\
&F_n \text{ is bounded in } L^\infty(0,T_0;H^2(\Omega)^{d\times d})\text{ independently of } n,
\shortintertext{and,}
&(u_n)_t \text{ is bounded in } L^\infty(0,T_0;H) \cap L^2(0,T_0;V) \text{ independently of } n,\\
& (\phi_n)_t \text{ is bounded in } L^\infty(0,T_0;H^1(\Omega)) \cap L^2(0,T_0;H^3(\Omega)) \text{ independently of } n,\\
&(F_n)_t \text{ is bounded in }  L^\infty(0,T_0;L^2(\Omega)^{d\times d}) \text{ independently of } n.
\end{align*}
Hence $(u_n,\phi_n,F_n)$ converges to $(u,\phi,F)$, for some subsequence $n \rightarrow \infty$, in the following sense:
\begin{align*}
&u_n \rightharpoonup u \text{ weak star in }L^\infty(0,T_0;V) \text{ and weakly in }L^2(0,T_0;H^3(\Omega)^d),\\
&(u_n)_t \rightharpoonup u_t \text{ weak star in }L^\infty(0,T_0;H) \text{ and weakly in }L^2(0,T_0;V),\\
&u_n \rightarrow u \text{ a.e. in } \Omega \times (0,T_0) \text{ and in } L^2(0,T_0;D(A)),\\
&\phi_n \rightharpoonup u \text{ weak star in }L^\infty(0,T_0;H^2(\Omega)) \text{ and weakly in }L^2(0,T_0;H^4(\Omega)),\\
&(\phi_n)_t \rightharpoonup \phi \text{ weak star in }L^\infty(0,T_0;H^1(\Omega)) \text{ and weakly in }L^2(0,T_0;H^3(\Omega)),\\
&\phi_n \rightarrow \phi \text{ a.e. in }\Omega \times (0,T_0) \text{ and in } L^2(0,T_0;H^{3+\epsilon}(\Omega)) \quad \forall\e \in [0,1),\\
&F_n \rightharpoonup F \text{ weak star in }L^\infty(0,T_0;H^2(\Omega)^{d\times d}),\\
&(F_n)_t \rightharpoonup F_t \text{ weak star in }L^\infty(0,T_0;L^2(\Omega)^{d\times d}).
\end{align*}
Here we have also used the Aubin-Lions compactness theorem and the spaces $H^{3+\epsilon}(\Omega))$ are defined by interpolation for $0<\e<1$ (see e.g. \cite{Tem97}). We can now pass to the limit $n \rightarrow \infty$ in \eqref{galerkin_eqn} using a standard argument.
 In order to obtain continuity in time results for $u,\phi,F$,  we use the following result  (Theorem 2.3 in \cite{LM72}): Given $X$ and $Y$ two Hilbert spaces, the space $W_{2,2}:=\{v \in L^2(0,T;X), v_t \in L^2(0,T;Y)\}$ is continuously embedded in $C([0,T];[X,Y]_{\frac12})$ where $[X,Y]_{\frac12}$ is the interpolation space of order $\frac12$ of $X$ and $Y$. This implies that,
	\begin{align*}
	u \in C([0,T_0]; \bH^2(\Omega)), \,\, \phi \in C([0,T_0];H^3(\Omega)), \,\, F\in C([0,T_0];H^1(\Omega)^{d\times d}).
	\end{align*}
Since $u$ takes values in $\bH^2(\Omega)$, and the norm of $D(A)$ is equivalent to that of $\bH^2(\Omega)$, we also see that $u$ belongs to $C([0,T_0];D(A))$.\\
 This completes the proof for the existence part of Theorem \ref{main}.\\

\begin{proof}[Proof of uniqueness]
We will now prove the uniqueness of strong solutions to the equations \eqref{pde}-\eqref{bc}.

Assume that there exist two solutions $(u_1,\phi_1,F_1)$ and $(u_2,\phi_2,F_2)$ to the equations \eqref{pde}-\eqref{bc} in the class mentioned in \eqref{strong}. Let $(u,\phi,F):=(u_1,\phi_1,F_1)-(u_2,\phi_2,F_2)$. Then we can see that $(u,\phi,F)$ solves the following difference equations
\begin{equation}\label{diff}
\begin{split}
&\rho(u_t + u_1\cdot \nabla u + u \cdot \nabla u_2 ) -\nabla \cdot (\eta(\phi_1)\nabla u+(\eta(\phi_1) -\eta(\phi_2))\nabla u_2=\\
&\qquad \qquad -\nabla \cdot (\nabla \phi_1 \otimes \nabla \phi + \nabla \phi_2\otimes \nabla\phi ) \\
&\qquad \qquad +\lambda_e \nabla \cdot ((1-\phi_1)(F_1F^T+FF_2^T)-\phi(F_2F_2^T-I))\\
&\qquad \qquad -\eta(\phi_1)\frac{(1-\phi_1)u_1}{\kappa(\phi_1)}+\eta(\phi_2)\frac{(1-\phi_2)u_2}{\kappa(\phi_2)},\\
& \nabla \cdot u=0,\\
& F_t + u_1 \cdot \nabla F + u \cdot \nabla F_2=\nabla u_1 F + \nabla uF_2,\\
&\phi_t+ u_1\cdot \nabla \phi + u\cdot \nabla \phi_2=
\tau\Delta \mu,\\
&\mu= -\lambda \Delta\phi + \lambda\gamma(f'(\phi_1)-f'(\phi_2))-\frac{\lambda_e}{2}(\tr{F_1F^T}+\tr{FF_2^T}).
\end{split}
\end{equation}
Next we test \eqref{diff}$_1$ by $u$, \eqref{diff}$_3$ by $F$, \eqref{diff}$_4$ by $\phi$, and then sum them up to obtain,
\begin{equation*}\begin{split}
\frac12 \frac{\ud }{\ud t}\int_\Omega |u|^2&+ |F|^2 + |\phi|^2 \ud x + \eta(\phi_1)|\nabla u|^2 + \tau\lambda| \Delta \phi|^2=-(u \cdot \nabla u_2,u) \\
&+(\nabla\cdot(\eta(\phi_1) -\eta(\phi_2))\nabla u_2, u) -(\nabla \phi_1 \otimes \nabla \phi + \nabla \phi_2\otimes \nabla\phi,\nabla u ) \\
&  +\lambda_e  ((1-\phi_1)(F_1F^T+FF_2^T)-\phi(F_2F_2^T-I),\nabla u)\\
&-\left(\eta(\phi_1)\frac{(1-\phi_1)u_1}{\kappa(\phi_1)}-\eta(\phi_2)\frac{(1-\phi_2)u_2}{\kappa(\phi_2)},u\right)\\
& -\tr(u\cdot \nabla F_2,F^T) + \tr(\nabla u_1F + \nabla u F_2,F^T)\\
&-(u \cdot \nabla  \phi_2,\phi)-(\nabla(f'(\phi_1)-f'(\phi_2)),\nabla \phi)+(\frac{\lambda_e}{2}\nabla( \tr{F_1F^T}+\tr{FF_2^T}),\nabla \phi).
\end{split}\end{equation*}
Below we will treat the terms, appearing on the right hand side above, that are more difficult to handle and leave the rest to the reader since the argument follows closely the a priori estimates found earlier. Observe that using interpolation inequalities we obtain,
\begin{align*}
|(u \cdot \nabla u_2,u)| &\leq \frac{\alpha}{16}|\nabla u|^2 + C|u_2|_2^2|u|^2.
\end{align*}
The terms $(u\cdot \nabla \phi_2,\phi)$ and $\tr(u\cdot \nabla F_2,F^T)$ are treated similarly. Next, using \eqref{equivnorm}, \eqref{avs} and the generalized Poincar\'e inequality we also have\begin{align*}
|(\nabla \phi_1 \otimes \nabla \phi + \nabla \phi_2\otimes \nabla\phi,\nabla u )| &\leq C(|\phi_1|_3+ |\phi_2|_3)|\phi|^{\frac12}|\phi|_2^{\frac12}|\nabla u|\\
& \leq  \frac{\alpha}{16}|\nabla u|^2+\frac{\tau\lambda}{16}| \Delta\phi|^2 + C(|\phi_1|_3+ |\phi_2|_3)^4|\phi|^2.
\end{align*}
Similarly,
\begin{align*}
&|((1-\phi_1)(F_1F^T+FF_2^T)-\phi(F_2F_2^T-I),\nabla u)| \lesssim |\phi_1|_2(|F_1|_2 + |F_2|_2)|F||\nabla u| \\
& \hspace{3.2in}+ |\phi||F_2|_2^2|\nabla u|\\
&\hspace{2in}  \leq  \frac{\alpha}{16}|\nabla u|^2 + C|\phi_1|^2_2(|F_1|_2 + |F_2|_2)^2|F|^2 +C|\phi|^2|F_2|_2^4.
\end{align*}
Observe next that,
\begin{align*}
|\tr(\nabla u_1 F + \nabla uF_2 ,F)|&\leq |u_1|_3|F|^2 + |\nabla u||F_2|_2|F|\\
& \leq \frac{\alpha}{16}|\nabla u|^2 + C(|u_1|_3+|F_2|^2_2)|F|^2,
\end{align*}
Next,
\begin{align*}
(\nabla(f'(\phi_1)-f'(\phi_2)),\nabla \phi) &= (f''(\phi_1)\nabla \phi -(f''(\phi_1)-f''(\phi_2)) \nabla \phi_2,\nabla \phi).
\end{align*}
Observe that,
\begin{align*}
(f''(\phi_1)\nabla,\nabla \phi) &\geq -C|\nabla \phi|^2,
\shortintertext{and since $f''$ is quadratic,}
|(f''(\phi_1)-f''(\phi_2) \nabla \phi_2,\nabla \phi)| &\leq C|\phi||\phi_2|_2|\phi|_2(|\phi_1+ \phi_2|_2)\\
&\leq \frac{\tau\lambda}{16}|\Delta\phi|^2 + C|\phi|^2|\phi_2|^2_2(|\phi_1+ \phi_2|^2_2).
\end{align*}
{For the final term, observe that,}
\begin{align*}
|(\nabla ( \tr{F_1F^T}+\tr{FF_2^T}),\nabla \phi)|
& = |( \tr{F_1F^T}+\tr{FF_2^T},\Delta \phi)|\\
& \lesssim (|F_1|_2| +|F_2|_2 )|F||\phi|_2\\
& \leq \frac{\tau\lambda}{16}|\Delta \phi|^2 + C(|F_1|_2| +|F_2|_2)^2|F|^2.
\end{align*}
Combining all the estimates derived above, we obtain,
\begin{equation}
\begin{split}
\frac12 \frac{\ud }{\ud t}\int_\Omega |u|^2+ |F|^2 +& |\phi|^2 \ud x  \leq C\mathcal{G}(|u|^2+ |F|^2 + |\phi|^2),
\end{split}
\end{equation}
where $\mathcal{G}=(|u_2|_2^2 + |u_1|_3+ |\phi_1|^4_3+ |\phi_2|^4_3 + |F_1|_2^4 + |F_2|^4)$.
Observing that $\mathcal{G} \in L^1(0,T_0)$, an application of the Gr\"onwall inequality gives us the desired uniqueness result and thus we conclude the proof of Theorem \ref{main}.
\end{proof}
\phantom\qedhere 
\end{proof}
\section{Acknowledgments}
This work was supported by the Research Fund of Indiana University.

\bibliography{chns_biblio}
\bibliographystyle{plain}
\end{document}